\documentclass{amsart}

\usepackage{amssymb, amsmath,  amsfonts,mathrsfs,graphicx }
\usepackage{color }
\usepackage[all,cmtip]{xy}

\newtheorem{thm}{Theorem}[section]
\newtheorem{prop}[thm]{Proposition}
\newtheorem{lemma}[thm]{Lemma}

\newtheorem{cor}[thm]{Corollary}
\newtheorem{defn}[thm]{Definition}
\newtheorem{example}[thm]{Example}
\newtheorem{remark}[thm]{Remark}

\newtheorem{qus}[thm]{Question}

 \numberwithin{equation}{section}

\begin{document}{\allowdisplaybreaks[4]

\title{On Frobenius-Perron dimension}


\author{Changzheng Li}
 \address{School of Mathematics, Sun Yat-sen University, Guangzhou 510275, P.R. China}
\email{lichangzh@mail.sysu.edu.cn}

\author{Ryan M. Shifler}
\address{ Department of Mathematics, Henson Science Hall, Salisbury University, Salisbury MD 21801 USA}
\email{rmshifler@salisbury.edu}
\author{Mingzhi Yang}
 \address{School of Mathematics, Sun Yat-sen University, Guangzhou 510275, P.R. China}
\email{yangmzh8@mail2.sysu.edu.cn}

\author{Chi Zhang}
 \address{Department of Mathematics, Caltech, 1200 East California Boulevard, Pasadena, CA 91125}
\email{czhang5@caltech.edu}
.

\thanks{}

\date{
      }




\begin{abstract}
 We propose a notion of Frobenius-Perron dimension for certain free $\mathbb{Z}$-modules of infinite rank and compute it for the $\mathbb{Z}$-modules of finite dimensional complex representations of unitary groups with nonnegative dominant weights. The definition of Frobenius-Perron dimension that we are introducing naturally generalizes the well-known Frobenius-Perron dimension on the category of finite dimensional complex representations of a finite group.
 \end{abstract}

\maketitle

\section{Introduction}

The well-known Frobenius-Perron theory (see e.g. \cite{BePl}) concerns eigenvalues and eigenvectors of  nonnegative square matrices of finite order, and has applications to many areas of research in mathematics.
As one remarkable property in Frobenius-Perron theory, the spectral radius of  an $N\times N$ irreducible nonnegative matrix is a simple eigenvalue of the matrix.
The notion of the Frobenius-Perron dimension was motivated by that
of the index of a subfactor \cite{Jone}.
It was first defined  for commutative fusion rings   by
Fr\"ohlich and Kerler \cite{FrKe} as functions satisfying certain properties. The theory of Frobenius-Perron dimensions for general fusion rings and categories was developed by Etingof, Nikshych and Ostrik \cite{ENO};   for some other cases, we refer to \cite{ Etin,  EtOs, CGWZ3,Etin2, Wick} and references therein. The key ingredient here is to define a function on
 a nice $\mathbb{Z}_+$-ring of finite rank whose value at a basis element is the spectral radius of the induced linear operator as guaranteed by the Frobenius-Perron theory.

 As a typical example, the Frobenius-Perron dimension ${\rm FPdim}$ of the Grothendieck ring $\mathbf{\rm Gr}({\rm Rep}(G_{\rm fin}))$ of the category of finite dimensional complex representations of a finite group $G_{\rm fin}$ is the  ring homomorphism
 $${\rm FPdim}: \mathbf{\rm Gr}({\rm Rep}(G_{\rm fin})) \to \mathbb{C}\quad\mbox{ with }\quad  {\rm FPdim}([V_{\rm fin}])=\dim V_{\rm fin}$$
 for any irreducible representation $V_{\rm fin}$ of $G_{\rm fin}$. There are infinitely many  isomorphism classes in
  the category ${\rm Rep}(G)$ of
 finite dimensional complex representations of a  compact Lie group (or equivalently, a reductive complex algebraic group) $G$ (or its Lie algebra ${\rm Lie}(G)$ alternatively). This, however, is one of the most important targets in representation theory. It is therefore quite natural to ask the following question.
\begin{qus}\label{ques1}
  Is there an appropriate  notion of Frobenius-Perron dimension ${\rm FPdim}$ for the Grothendieck ring
   $\mathbf{\rm Gr}({\rm Rep}(G))$, with  which ${\rm FPdim}: \mathbf{\rm Gr}({\rm Rep}(G)) \to \mathbb{C}$ is a function satisfying the following properties?
   \begin{enumerate}
     \item ${\rm FPdim}([V])=\dim V$
  for any  irreducible representation $V\in {\rm Rep}(G)$;
     \item ${\rm FPdim}$ is a ring homomorphism.
   \end{enumerate}
  \end{qus}

In the present paper, we propose a notion of Frobenius-Perron dimension  for   $\mathbb{Z}_+^\bullet$-rings, generalizing the standard one for $\mathbb{Z}_+$-ring of finite rank.  As we will see in Definitions \ref{defZ} and \ref{defmain}, we define the generalized Frobenius-Perron dimension to be the limit of the standard one for a $\mathbb{Z}_{\geq 0}$-filtration of $\mathbb{Z}_+$-rings of finite rank. By the definition, the expected property (1) in Question \ref{ques1} is not necessarily satisfied a priori.    Finite dimensional irreducible representations $\mathbb{S}_\lambda(V)$ of $G=U(k)$ are indexed by decreasing sequences $\lambda=(\lambda_1, \cdots, \lambda_k)$ of integers. The irreducible representations $\mathbb{S}_\lambda(V)$ with $\lambda_k\geq 0$  {generate a subcategory ${\rm Rep}(U(k))_{+}$}
of ${\rm Rep}(U(k))$.
 We  answer Question \ref{ques1} for unitary groups $U(k)$ in the expected way, yet partially in  the following sense.

 \begin{thm}
   There is a generalized  Frobenius-Perron dimension ${\rm FPd}^\bullet$ of the $\mathbb{Z}_+^\bullet$-ring $\mathbf{\rm Gr}({\rm Rep}(U(k))_+)$, given by the {$\mathbb{Z}_+^\bullet$-ring homomorphism}
  $${\rm FPd}^\bullet: \mathbf{\rm Gr}({\rm Rep}(U(k))_+) \to \mathbb{C}; \quad  {\rm FPd}^\bullet([\mathbb{S}_\lambda(V)])=\dim \mathbb{S}_\lambda(V).$$
 \end{thm}
\noindent We will restate the above conclusion in Theorem \ref{thmmain3} in a more precise way.

On one hand, the idea of our generalized notion of Frobenius-Perron dimension is natural. The free
$\mathbb{Z}$-module $\mathbf{\rm Gr}({\rm Rep}(U(k))_+)$ admits a natural $\mathbb{Z}_{\geq 0}$-filtration
of $\mathbb{Z}$-modules ${\rm Rep}(U(k))_\ell$ generated by isomorphism classes $[\mathbb{S}_\lambda(V)]$ with $\ell\geq \lambda_1\geq\lambda_2\geq \cdots \geq \lambda_k\geq 0$.
Each data ${\rm Rep}(U(k))_\ell$ is of finite rank, and  inherits a $\mathbb{Z}_+$-ring structure  induced from the ``natural" tensor product with the help of natural projections. On the other hand, the subtle point here is that we need to consider the ``quantum version" of the tensor product instead. Slightly more precisely, we need to consider the fusion ring structure of ${\rm Rep}(U(k))_\ell$, called the Verlinde algebra (named after Erik Verlinde \cite{Verl}) at level $(\ell, k+\ell)$.
Those  isomorphism classes $[\mathbb{S}_\lambda(V)]$ with $\ell\geq \lambda_1\geq  \cdots \geq \lambda_k\geq 0$ form a $\mathbb{Z}_+$-basis, with the structure constants counting the dimension of the space of sections of appropriate line bundles over the moduli space of semi-stable parabolic bundles, or equivalently, counting the dimension of certain vector spaces of conformal blocks as studied in the physics literature.  (See \cite{Witt, Agni, Beau, Belk} for more details.)

We do the computation by using Witten's   remarkable ring isomorphism  \cite{Witt, Agni,Belk}
$${\rm Rep}(U(k))_\ell\to QH^*(Gr(k, k+\ell))|_{q=1}; \quad [\mathbb{S}_\lambda(V)]\mapsto [X_\lambda]. $$
Here $QH^*(Gr(k, k+\ell))$ denotes   the (small) quantum cohomology   of the complex Grassmannian $Gr(k, k+\ell)$. {The $\mathbb{Z}_+$-ring $QH^*(Gr(k, k+\ell))|_{q=1}$ has a standard $\mathbb{Z}_+$-basis of  Schubert classes $[X_\lambda]$.} Thanks to the above isomorphism,   the computation for the Verlinde algebra can be translated to that for the spectral radius of linear operators on $QH^*(Gr(k, k+\ell))|_{q=1}$ induced by the Schubert classes. The latter objects are of significant interest in algebraic geometry and mirror symmetry in their own right. For instance, Galkin, Golyshev and Iritani \cite{GGI} proposed Gamma conjectures I and II for Fano manifolds $X$, together with the underlying conjecture $\mathcal{O}$ that concerns the spectral radius of the linear operator $\hat c_1(X)$ on $QH^*(X)|_{\mathbf{q}=\mathbf{1}}$ induced by the first Chern class $c_1(X)$. The conjecture $\mathcal{O}$ for general homogeneous varieties (of which $Gr(k, k+\ell)$ is a special case) was verified by Cheong and the first named author \cite{ChLi} by using Frobenius-Perron theory. There is another conjecture proposed by Galkin \cite{Galk}, which concerns the lower bound of the spectral radius of linear operator $\hat c_1(X)$. This was recently verified for  complex Grassmannian \cite{ESSSW} and for Lagrangian and orthogonal Grassmannians \cite{ChHa}. Our generalized notion of Frobenius-Perron dimension was actually motivated by the studying the properties (especially \textbf{Theorem \ref{thmmain}}) of the spectral radius of the linear operators induced from the basis of Schubert classes.
  As a byproduct, we provide lower bounds in \textbf{Theorem \ref{thmmain2}} for special radius of these linear operators.

We finally remark that  
 Question \ref{ques1} is related with  the amenability \cite{Popa, LoRo, HiIz} of dimension problems in representation theory and operator algebras. The study of analytic properties in the theory of rigid $C^*$-tensor categories $\mathcal{C}$  goes back to   \cite{GLR,DoRo}.    The positive dimension $d_{\rm min}$ function on $\mathcal{C}={\rm Rep}(G)$, obtained by taking standard minimal solutions to the conjugate equations, is amenable and the value $d_{\rm min}([V])$ at an irreducible representation $V$  equals $\dim V$ (see e.g. \cite[\S 2.7]{NeTu}). Our generalized Frobenius-Perron dimension provides    an algebraic approach to the realization of $d_{\rm min}$ in the case $G=U(k)$. Actually, by the name ``Frobenius-Perron" we have  assumed  an implicit condition in  Question \ref{ques1}  that ${\rm FPdim}([V])$ is somehow given by a spectral radius of a linear operator.
 The fusions rings of complex simple Lie algebras of general Lie types can be defined \cite{Beau}. It would be interesting to investigate whether there are desired  answers to Question \ref{ques1} for these general cases.

 The paper is organized as follows. In section 2, we review the standard notion of Frobenius-Perron dimension of a $\mathbb{Z}_+$-ring of finite rank and generalize it to that of certain $\mathbb{Z}_+$-ring of infinite rank. In section 3, we study properties of the spectral radius of linear operators on $QH^*(Gr(k, n))|_{q=1}$ induced by   Schubert classes. In section 4, we compute the Frobenius-Perron dimension for the polynomial representation ring of unitary groups.

\section{Frobenius-Perron dimension}
We mainly follow  \cite{EGNO} for the standard notion of Frobenius-Perron dimension.
A basis  $\{\beta_i\}_{i\in I}$ of a ring  which is free as a $\mathbb{Z}$-module is called a  $\mathbb{Z}_{+}$-basis if
 $\beta_i\cdot \beta_j=\sum_{r\in I} c_{ij}^r\beta_r$ with $c_{ij}^r\in \mathbb{Z}_{\geq 0}$ for any $i, j, r$.
Let $A$ be a $\mathbb{Z}_{+}$-ring, namely  a ring with identity 1 together with a
fixed $\mathbb{Z}_{+}$-basis $\{\beta_i\}_{i\in I}$.
  Each $\beta_i$ induces a linear operator $\hat {\beta}_i: A\to A; \gamma\mapsto \beta_i \cdot \gamma$.
If $I$ is  finite, then for   $i\in I$, the well-known   Frobenius-Perron theory on nonnegative matrices
 (see e.g. \cite[\S 2, Theorem 1.1]{BePl}) ensures that the spectral radius $\rho(\hat{\beta}_i)$ of $\hat \beta_i$, 
 $$\rho(\hat \beta_i):= \max\{|c|\mid c\mbox{ is an eigenvalue of } \hat\beta_i\}\in \mathbb{R}_{\geq 0},$$
is an eigenvalue of the linear operator $\hat\beta_i$.
\begin{defn}\label{fpdimstd} Let $A$ be a $\mathbb{Z}_{+}$-ring of finite rank.
The function ${\rm{FPdim}}={\rm{FPdim}}_A$,   
 $${\rm{FPdim}}: A\to \mathbb{C};\quad {\rm{FPdim}}\left(\sum\nolimits_ia_i\beta_i \right):=\sum\nolimits_ia_i\rho(\hat \beta_i),$$
  is called the \textbf{Frobenius-Perron dimension} of $A$.
\end{defn}
\noindent Furthermore, if $A$ is   transitive and unital, then ${\rm FPdim}: A\to \mathbb{C}$ is a ring homomorphism, so that the above notion  is consistent with that for commutative Fusion rings  introduced 
in \cite{FrKe} .
The theory of Frobenius-Perron dimensions for general fusion rings and categories was developed in \cite{ENO}.

 \begin{example}
 Let $\mathcal{C} = {\rm Rep}(G_{\rm fin})$ be the category of finite dimensional complex representations
of a finite group $G_{\rm fin}$, and   $A={\rm Gr}\big({\rm Rep}(G_{\rm fin})\big)$ be its Grothendieck ring.  Then for any $V\in \mathcal{C}, {\rm FPdim}([V])=\dim_{\mathbb{C}}(V)$.
 \end{example}

In order to explore an answer to Question \ref{ques1}, we propose the following definition.

\begin{defn}\label{defZ}
     We call a pair  $\big(A, \{(A_r, B_r)\}_{r\in \mathbb{Z}_{\geq 0}}\big)$
   a   $\mathbb{Z}_+^\bullet$-ring (and simply denote it as $A$), if $A$ is a free $\mathbb{Z}$-module and the   family   $\{(A_r, B_r)\}_{r\in \mathbb{Z}_{\geq 0}}$ satisfy the following.
    \begin{enumerate}
      \item $\{(A_r,+)\}_{r}$ is a  $\mathbb{Z}_{\geq 0}$-filtration of $(A, +)$ of free $\mathbb{Z}$-modules of finite rank;
      \item $(A_r, +, \star_r)$ is a $\mathbb{Z}_{+}$-ring    with fixed $\mathbb{Z}_{+}$-basis $B_r$;
      \item $B_r\subset B_{r'}$, for any $  r, r'\in \mathbb{Z}_{\geq 0}$ with $r<r'$;
      \item $ B_0$ contains an element $1$ which is the   identity of $(A_r, \star_r)$ for all $r$.
    \end{enumerate}
    We call a map $F: A\to \mathbb{C}$ a  $\mathbb{Z}_+^\bullet$-ring homomorphism, if there exists a family
    $\{F_r: (A_r, +, \star_r)\rightarrow (\mathbb{C}, +, \cdot)\}_r$ of   ring homomorphisms in the usual sense such that $\lim\limits_{r\to+\infty}F_r(\alpha)=F(\alpha)$ for any $\alpha\in A$ (where  $F_r(\alpha):=0$ if $\alpha\notin A_r$ for conventions).
\end{defn}

\begin{defn}\label{defmain} Let  $\big(A, \{(A_r, B_r)\}_{r\in \mathbb{Z}_{\geq 0}}\big)$ be a  $\mathbb{Z}_+^\bullet$-ring.
 Assume  $\lim\limits_{r\to +\infty} {\rm FPdim}_{A_r}(\beta)$ exits and belongs to $\mathbb{R}$ for all $\beta\in \bigcup_{r=0}^\infty B_r$. We denote
    ${\rm{FPd}^\bullet}(\beta):= \lim\limits_{r\to +\infty} {\rm FPdim}_{A_r}(\beta)$, and call its linear extension
    ${\rm FPd}^\bullet: A\to \mathbb{C}$  the  \textbf{Frobenius-Perron dimension} of $\big(A, \{(A_r, B_r)\}_{r\in \mathbb{Z}_{\geq 0}}\big)$
   if ${\rm FPd}^\bullet$ is a $\mathbb{Z}_+^\bullet$-ring homomorphism.
   \end{defn}

  \begin{example}
    Let $A$ be a $\mathbb{Z}_+$-ring  of finite rank with fixed $\mathbb{Z}_+$-basis $B$. Then $A$ is a $\mathbb{Z}_+^\bullet$-ring with respect to the trivial $\mathbb{Z}_{\geq 0}$-filtration $\{A_r\}$ defined by $A_r=A$ and $B_r=B$ for any $r\geq 0$. In this case, a $\mathbb{Z}_+^\bullet$-ring homomorphism is a ring homomorphism in the usual sense, and the generalized notion of Frobenius-Perron dimension
     of $A$ coincides with the standard one as defined in Definition \ref{fpdimstd}
   \end{example}
\begin{remark}
    The underlying $\mathbb{Z}$-module $A$ of a  $\mathbb{Z}_+^\bullet $-ring is not necessarily a ring a priori.
\end{remark}

\begin{example}\label{exainduced}
Let $A$ be a $\mathbb{Z}_+$-ring  equipped with  $\{(A_r, B_r)\}_r$ such  that {\upshape (a) } $1_A\in B_0$; {\upshape (b)} $B_r$ is an additive basis of $A_r$; and {\upshape (c)}  conditions (1) and (3) in Definition \ref{defZ} are satisfied. Then we can equip $A$   with a $\mathbb{Z}_+^\bullet$-ring structure, by considering the induced $\mathbb{Z}_+$-ring structure on $(A_r, +, \circ_r)$ defined by
  $$\alpha\circ_r \beta:= \pi_r (\alpha\circ \beta), \quad \forall \alpha, \beta\in A_r.$$
Here  $\pi_r: A\to A_r$ is the morphism of $\mathbb{Z}$-modules given by   $\pi_r(\beta)=\beta$ if $\beta \in B_r$, or $0$ if $\beta\in \bigcup_sB_s\setminus B_r$.
In this case, the limit
    $\lim\limits_{r\to +\infty} {\rm FPdim}_{A_r}(\beta)$ exists (possibly equal to $+\infty$) for any $\beta\in B$, by
     Frobenius-Perron theory (see e.g. \cite[\S 2, Corollary 1.6]{BePl}).
   \end{example}

\section{Quantum cohomology of Grassmannians}
Our generalized notion of Frobenius-Perron dimension is motivated by the study of the  quantum cohomology of the complex Grassmannian $Gr(k, n)=\{W\leqslant \mathbb{C}^n\mid \dim W =k\}$, where $k, n\in \mathbb{Z}_{>0}$ with $k<n$.  The readers with a strong interest in the computation of generalized Frobenius-Perron dimension for polynomial representation ring of unitary groups, can skip this section.

 Set $E^{(i)}_i=\mathbb{C}^i$ and $E^{(j)}_{i}=E^{(j)}_{i-1}\times\{0\}$ for any $i\leq j$.
Then $E^{(n)}_\bullet=\{E^{(n)}_1\leqslant E^{(n)}_2\leqslant\cdots \leqslant E^{(n)}_n\}$ is a complete flag in $E^{(n)}_n=\mathbb{C}^n$.
Let $\mathcal{P}_{k}(n)=\{\lambda=(\lambda_1, \cdots, \lambda_k)\in \mathbb{Z}^k\mid  n-k\geq \lambda_1\geq \cdots \geq \lambda_k\geq 0\}$.
For $\lambda\in \mathcal{P}_{k}(n)$, we denote $|\lambda|=\sum_{i=1}^k\lambda_i$ and $\lambda^\vee=(n-k-\lambda_k, \cdots, n-k-\lambda_1)$.

The  complex Grassmannian $Gr(k, n)$ has a class of closed subvarieties, called Schubert subvarieties that are defined by
 $$X_\lambda=X_\lambda(E_\bullet^{(n)})=\{W\in Gr(k, n)\mid \dim (W\cap E_{n-k+i-\lambda_i}^{(n)})\geq i, i=1,\ldots, k\}$$
where $X_\lambda$ is of codimension $|\lambda|$ for any partition $\lambda\in \mathcal{P}_k(n)$.
The cohomology classes $\{[X_\lambda]\in H^{2|\lambda|}(Gr(k, n), \mathbb{Z})\}$ form
a  $\mathbb{Z}_+$-basis of the integral cohomology ring $H^*(Gr(k, n))=H^*(Gr(k, n),\mathbb{Z})$, which is torsion free and of rank ${n\choose k}$.

The (small) quantum cohomology  $QH^*(Gr(k, n))=(H^*(Gr(k, n))\otimes \mathbb{Z}[q], \star)$ is  a commutative ring with the quantum product for any $\lambda, \mu\in \mathcal{P}_k(n)$ defined by
  $$[X_\lambda]\star [X_\mu]=\sum_{\nu\in \mathcal{P}_k(n), d\in \mathbb{Z}_{\geq 0}} N_{u, v}^{w, d}[X_\nu] q^d.$$
Here the Schubert structure constant $N_{u, v}^{w, d}$, known as a genus 0, 3-point Gromov-Witten invariant, counts the number of holomorphic maps $f: \mathbb{P}^1\to Gr(k, n)$ of degree $d$ with $f(0)\in X_\lambda,  f(1)\in g\cdot X_\mu$ and $f(\infty)\in g'\cdot X_{\nu^\vee}$ for generic (fixed) $g, g'\in GL(n, \mathbb{C})$. In particular, $N_{\lambda, \mu}^{\nu, d}$ is a non-negative integer for any $d$, and it vanishes for any sufficient large $d$.

 For any partition $\lambda\in\mathcal{P}_k(n)$, the quantum product by the Schubert class $[X_{\lambda}]$ induces a linear operator
$$\widehat{[X_\lambda]}: QH^*(Gr(k, n))|_{q=1}\longrightarrow QH^*(Gr(k, n))|_{q=1};\, \beta\mapsto [X_\lambda]\star \beta|_{q=1}.$$
Its eigenvalues and eigenvectors have been well studied by Rietsch \cite{Riet} in terms of Schur functions and primitive $n$-th roots of unity. In particular,
the following lemma follows immediately from Theorem  8.4 (1) and section 11 of \cite{Riet}
\begin{lemma}\label{lemspec}
  Let $\lambda\in \mathcal{P}_k(n)$, and  let $\rho_{k, \lambda}(n)$ denote the spectral radius of $\widehat{[X_\lambda]}$, namely
  $$\rho_{k, \lambda}(n)=\max\{|c|\mid c \mbox{ is an eigenvalue of the operator } \widehat{[X_\lambda]}\mbox{ on } QH^*(Gr(k, n))|_{q=1}\}.$$
Denote {\upshape $\mbox{hl}(i, j)=\lambda_i+\lambda_j^t-i-j+1$}, where $\lambda^t=(\lambda_1^t, \cdots, \lambda_{n-k}^t)$ denotes the transpose partition of
  $\lambda$.  Then we have
{\upshape \begin{equation}\label{spec}
   \rho_{k, \lambda}(n)={\prod_{(i,j)\in \lambda} \sin\big((k-i+j){\pi\over n}\big)\over \prod_{(i,j)\in \lambda} \sin\big(\mbox{hl}(i,j){\pi\over n}\big)},
\end{equation}
}
\end{lemma}
\begin{example} The number {\upshape $\mbox{hl}(i, j)$} equals the \textit{hook length} of the box labeled by $(i, j)$ in the Young diagram of the partition $\lambda$. The following figure shows the case $\lambda=(6,4,2,1)\in \mathcal{P}_4(11)$, for which $\lambda^t=(4,3,2,2,1,1,0)$.
 
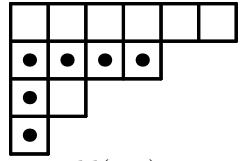
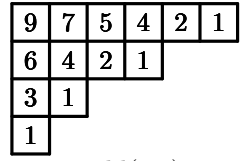
\begin{figure}[h]
     \includegraphics[scale=1]{hlength.1}
  \end{figure}
\end{example}
\begin{remark}
The number $\rho_{k,(1,0,\cdots,0)} (n)$ is the length of the $k$-th diagonal of a regular $n$ sided polygon with unit side length.
\end{remark}

Notice that there are   $|\lambda|$ boxes in the $i$-th row of the  {Young diagram} of the partition $\lambda$.
The sequence  $\{k-i+j\}_{(i, j)\in \lambda}$ of length $|\lambda|$ can be reordered as $\{a_1,\cdots, a_{|\lambda|}\}$ such that
$a_r= k-i+\lambda_i-j+1$ correspondes to  a unique $(i, j)\in \lambda$. For the same $(i, j)$, we denote $b_r=\mbox{hl}(i, j)$. Then $a_r\geq b_r$ for all $1\leq r\leq |\lambda|$, since
   $$k-i+\lambda_i-j+1\geq \lambda_j^t+\lambda_i-i-j+1= \mbox{hl}(i, j), \mbox{ for any }(i, j)\in \lambda.$$
By using formula $\eqref{spec}$, we can define a smooth function $\rho_{k, \lambda}: \mathbb{R}^+\to \mathbb{R}$. The next property is an immediate consequence of the definitions of $a_r$'s and $b_r$'s.
\begin{cor}\label{corlambdaeq} For any $\lambda\in \mathcal{P}_k(n)$, the following are equivalent.
  {\upshape $$\mbox{i) } \rho_{k, \lambda}(x)=1; \quad \mbox{ii) } a_r=b_r \mbox{ for all } 1\leq r\leq |\lambda|;\quad \mbox{iii) }\lambda_1=\lambda_2=\cdots=\lambda_k.$$}
\end{cor}
\begin{lemma}[Hook-length formula; see page 61 of \cite{Robi}]\label{lemhook}
   Let $\mathbb{S}_\lambda(V)$ denote the irreducible representation of $U(k)$ associated to $\lambda$. Then
    {\upshape $$\dim \mathbb{S}_\lambda(V)=   {\prod_{(i,j)\in \lambda}  (k-i+j) \over \prod_{(i,j)\in \lambda} \mbox{hl}(i,j)}.
$$}
\end{lemma}

\begin{thm}\label{thmmain} For   $\lambda\in\mathcal{P}_{k}(n)$, we have
     $$\lim\limits_{x\to +\infty} \rho_{k, \lambda}(x)= \dim \mathbb{S}_\lambda(V).$$

\end{thm}
\begin{proof}
   The statement follows immediately from  and Lemmas \ref{lemspec} and \ref{lemhook}.
\end{proof}
\begin{thm}\label{thmmain2} For   $\lambda\in\mathcal{P}_{k}(n)$ with $\lambda_i\neq \lambda_j$ for some $i\neq j$, the following hold.
  \begin{enumerate}
      \item We have the inequality $$\rho_{k, \lambda}(n)\geq \dim \mathbb{S}_{\lambda}(V)\prod_{(i, j)\in \lambda} \left(1- \frac{\pi^2 (k-i+j)^2}{6n^2} \right).$$

     \item  The   function  $\rho_{k, \lambda}(x)$ is strictly increasing on the interval $(k+\lambda_1-1,+\infty)$, and is         concave down when $x$ is sufficiently large.
  \end{enumerate}
\end{thm}
\begin{proof}
  We leave the details in the appendix.
\end{proof}
\begin{remark}\label{rmkk}
 The above lower bound of $\rho_{k, \lambda}(n)$ works for any $1\leq k<n$ and $\lambda\in \mathcal{P}_k(n)$.  There is a natural isomorphism $QH^*(Gr(k, n))\longrightarrow QH^*(Gr(n-k, n))$ that sends every Schubert class $[X_\lambda]$ for $Gr(k, n)$
 to the Schubert class $[X_{\lambda^t}]$ of $Gr(k, n)$ labeled by the transpose $\lambda^t$ of the partition $\lambda\in \mathcal{P}_k(n)$. As a consequence, we have
  $$\rho_{k, \lambda}(n)=\rho_{n-k, \lambda^t}(n)\geq \dim\mathbb{S}_{\lambda}(V^t) \prod_{(i, j)\in \lambda^t} \left(1- \frac{\pi^2 (n-k-i+j)^2}{6n^2} \right)$$
  { where $\mathbb{S}_{\lambda^t}(V^t)$ denotes the irreducible representation of $U(n-k)$ associated to $\lambda^t$.}
\end{remark}

\begin{example}\label{exak2}
   For $k=2$, we have $\dim \mathbb{S}_\lambda(V)=\lambda_1-\lambda_2+1$. Moreover, by  simplifying formula \eqref{spec}, we  obtain $\displaystyle \rho_{2, \lambda}(n)={\sin {(\lambda_1-\lambda_2+1)\pi\over n}\over \sin {\pi \over n}}$. It follows from Theorem \ref{thmmain2} (1) that
      $\rho_{2, \lambda}(n)\geq (\lambda_1-\lambda_2+1)\left(1-{\pi^2(\lambda_1+1)^2\over 6n^2}\right)$. In particular,
     $$\rho_{2, (1, 0)}(n)\geq 2\left(1-{4\pi^2\over 6 n^2}\right)>2\left(1-{ 3\over 2n}\right)={k}{n-k\over n}+{1\over n}\quad\mbox{for }n>4,$$
      $ \rho_{2, (1, 0)}(3)={\sin {2\pi\over 3}\over \sin {\pi\over 3}}=1={2}{3-2\over 3}+{1\over 3}$, and
        $ \rho_{2, (1, 0)}(4)={\sin {2\pi\over 4}\over \sin {\pi\over 4}}=\sqrt{2}>{2}{4-2\over 4}+{1\over 4}$.
\end{example}
\begin{example}
   For $\lambda=(1,0,\cdots, 0)$, we notice $\dim \mathbb{S}_\lambda(V)=k$ so that
      $$\rho_{k, \lambda}(n)\geq k\left(1-{\pi^2k^2\over 6n^2}\right) >{k}{n-k\over n}+{1\over n} \mbox{ for  } 3\leq k\leq {n\over 2} $$
      by  Theorem \ref{thmmain2} (1).
Together with Remark \ref{rmkk} and Example \ref{exak2}, this show that
$$n\rho_{k, (1, 0,\cdots, 0)}(n)\geq \dim Gr(k, n)+1=k(n-k)+1$$
with the equality holding if and only if $k=1$ or $n-1$.
This is exactly the proof of Galkin's lower bound conjecture   for all $Gr(k, n)$ as was given in \cite{ESSSW}.
\end{example}

\section{Frobenius-Perron dimension of polynomial representation ring of unitary groups}
     As in the introduction, we consider the Grothendieck ring ${\rm Gr}\big({\rm Rep}(U(k))_{+}\big)$ of the subcategory ${\rm Rep}(U(k))_{+}$ of finite dimensional complex representations of $U(k)$ generated by the isomorphism classes of irreducible representations $\mathbb{S}_\lambda(V)$ with $\lambda\in \bigcup_{r=0}^\infty\mathcal{P}_k(k+r)$, namely $\lambda=(\lambda_1,\cdots, \lambda_k)\in \mathbb{Z}^k$ satisfying $\lambda_1\geq \cdots\lambda_k\geq 0$. The Grothendieck ring $A={\rm Gr}\big({\rm Rep}(U(k))_{+}\big)$, also  referred   to as the polynomial representation ring of $U(k)$,
    is a $\mathbb{Z}_+$-ring of infinite rank. Moveover, it admits a standard $\mathbb{Z}_{\geq 0}$-filtration $\{A_r\}_r$ of free $\mathbb{Z}$-modules, with  $B_r:=\{[\mathbb{S}_\lambda(V)]\mid \lambda\in \mathcal{P}_k(k+r)\}$ being a basis of $A_r$ for each $r$.
 As to be described below, the  $\mathbb{Z}_{\geq 0}$-filtration $\{A_r\}_r$ can be equipped with different ring structures, for which $A$ becomes a $\mathbb{Z}_+^\bullet$-ring.
\subsection{Ring structures induced from the tensor product}

Consider the  morphism $\pi_{r}$ of $\mathbb{Z}$-modules given by  $\pi_{r}: A\to A_r$ with $\pi_{r}([\mathbb{S}_\lambda(V)])=
 [\mathbb{S}_\lambda(V)]$ if $\lambda\in \mathcal{P}_k(k+r)$, or $0$ otherwise.
Notice that the ring structure of $A$ is given by  tensor product of representations. This induces a natural ring structure $(A_r, \circ_r)$ by
  $$ [\mathbb{S}_\lambda(V)]\circ [\mathbb{S}_\mu(V)]=\pi_{r}([\mathbb{S}_\lambda(V)\otimes  \mathbb{S}_\mu(V)]),\quad \forall \lambda, \mu\in \mathcal{P}_k(k+r).$$
Then $(A_r, \circ_r)$ is a $\mathbb{Z}_+$-ring of finite rank with the identity  $\pi_{r}([\mathbf{1}])$, where
 $\mathbf{1}$ denote the trivial representation of $U(k)$, whose   isomorphic class  is the identity of  $A$.
  Consequently, the limit
  ${\rm FPd}^\bullet([\mathbb{S}_\lambda(V)])= \lim\limits_{r\to+\infty}{\rm FPdim}_{A_r}([\mathbb{S}_\lambda(V)])$  exists for any $\lambda$ with $\lambda_k\geq 0$ (for instance by \cite[\S 2, Corollary 1.6]{BePl}).  More is true: there is a   ring isomorphism  (see e.g. \cite[section 9.4]{Fult})
    $$(A_r, \circ_r)\longrightarrow (H^*(Gr(k, k+r)),\cup); [\mathbb{S}_\lambda(V)]\mapsto [X_\lambda].$$
  For any $\lambda\neq \mathbf{0}$, the cup product of any cohomology class by $[X_\lambda]$ increases the degree, and hence the operator $[X_\lambda]\cup$ is nilpotent. Therefore for any $r$ and any $\lambda\in \mathcal{P}_k(k+r)$, we have ${\rm FPd}^\bullet([\mathbb{S}_\lambda(V)])={\rm FPdim}_{A_r}([\mathbb{S}_\lambda(V)])=1$ if $\lambda=\mathbf{0}$, or $0$ otherwise. In a summary, we have the following.

  \begin{thm}
     The   $\mathbb{Z}$-module $A{={\rm Gr}\big({\rm Rep}(U(k))_{+}\big)}$ equipped with the family $\{(A_r, B_r)\}_r$ is a $\mathbb{Z}_+^\bullet$-ring.
     The family $\{{\rm FPdim}_{A_r}: A_r\to \mathbb{C}\}_r$ and the resulting in map ${\rm FPd}^\bullet: (A, \circ) \to \mathbb{C}$ are all trivial ring homomorphisms. In particular, ${\rm FPd}^\bullet$ is
     the  Frobenius-Perron dimension of the $\mathbb{Z}_+^\bullet$-ring $\big(A, \{(A_r, B_r)\}_r\big)$.
  \end{thm}

\subsection{Verlinde algebras}
 The free $\mathbb{Z}$-module $A_r$ can be equipped with the fusion ring structure $\star_r$ at level $(r, k+r)$, called the Verlinde algebra (of $U(k)$) at level $(r, k+r)$ in the physics literature. We   follow \cite{Belk} for the description of $\star_r$ below.

 The irreducible representation $\mathbb{S}_\lambda(V)$ of $U(k)$ restrict to the irreducible representation $V_{\bar \lambda}$ of $SU(k)$, where $\bar\lambda =(\lambda_1-\lambda_2,\cdots, \lambda_{k-1}-\lambda_k)$. The dual representation $V^*_{\bar \lambda}$ is  irreducible and hence given by $V_{\bar \lambda^*}$ for a corresponding dominant weight $\bar \lambda^*$.
The fusion ring of $SU(k)$ at level $r$, denoted by $R(SU(k))_r$, is an associated, commutative ring, defined by
  $$[V_{\bar \lambda}]*_r [V_{\bar \mu}]=\sum_{\bar \nu} N_0^{(r)}(\bar \lambda, \bar \mu, \bar \nu)[V_{\bar \nu^*}]$$
where
let $ N_0^{(r)}(\bar \lambda, \bar \mu, \bar \nu)$ denotes the dimension of the corresponding vector space of
conformal blocks for genus $0$ at level $r$ as described in \cite{Beau}.
The tensor product of $R(SU(k))_r$ and the fusion ring $R(U(1))_{(k+r)k}=\mathbb{Z}[x]/(x^{k(k+r)}-1)$ contains a unital subring
 $\tilde R\leqslant  R(SU(k))_r\otimes_\mathbb{Z}R(U(1))_{(k+r)k}$ spanned  by elements of the form $[V_{\bar \lambda}]\otimes x^a$ where $|\bar \lambda|\equiv a\,(\!\!\!\!\mod k)$. Denote by $\eta_\lambda:=(r+\lambda_k, \lambda_1,\cdots, \lambda_{k-1})$. The $\mathbb{Z}$-submodule $\mathcal{I}$ spanned by $\{[V_{\bar \eta_\lambda}]\otimes x^{a+k+r}-[V_{\bar \lambda}]\otimes x^a\mid |\bar \lambda|\equiv a\,(\!\!\!\!\mod k)\}$ is in fact an ideal of $\tilde R$. Then the fusion ring $(A_r, \star_r)$ (i.e. the fusion ring of $U(k)$ at level $(r, r+k)$) can be revealed as the quotient ring $\tilde R/\mathcal{I}$,
by identifying $[\mathbb{S}_\lambda(V)]\in A_r$ with $[V_{\bar \lambda}]\otimes x^{|\lambda|}+\mathcal{I}\in \tilde R/\mathcal{I}$. The following   remarkable property is  due to Witten \cite{Witt}  (see \cite{Agni, Belk} for mathematical proofs).
\begin{prop}\label{Witteniso}
   The natural isomorphism of $\mathbb{Z}$-modules $$ \Phi_r: (QH^*(Gr(k, k+r))|_{q=1}, \star) \longrightarrow (A_r, \star_r); \quad [X_\lambda]\mapsto [\mathbb{S}_\lambda(V)]$$
  is an isomorphism of rings.
\end{prop}
\noindent Combining the above descriptions  with Theorem \ref{thmmain}, we have the following.
   \begin{thm}\label{thmmain3}
      The   $\mathbb{Z}$-module $A={{\rm Gr}\big({\rm Rep}(U(k))_{+}\big)}$ equipped with the family $\{((A_r, \star_r), B_r)\}_r$ is a $\mathbb{Z}_+^\bullet$-ring.
    The  Frobenius-Perron dimension $\rm FPd^\bullet: A\to \mathbb{C}$ is well defined (which is a $\mathbb{Z}_+^\bullet$-ring homomorphism), and is given by    $${\rm FPd}^\bullet: A\to \mathbb{C};\quad {\rm FPd}^\bullet([\mathbb{S}_\lambda(V)])=\dim \mathbb{S}_\lambda(V).$$
  \end{thm}
  \begin{proof}
    Notice that $B_r=\{[\mathbb{S}_\lambda(V)]\mid \lambda \in \mathcal{P}_k(r)\}\subset B_{r+1}$ is a $\mathbb{Z}_+$-basis of $A_r$, and that $[\mathbf{1}]=[\mathbb{S}_{\mathbf{0}}(V)]\in A$ is the common identity of all $(A_r, \star_r)$. Therefore the first statement holds.

    Let $n=k+r$. By Proposition \ref{Witteniso} and Theorem \ref{thmmain},  we have
     $${\rm FPd}^\bullet ([\mathbb{S}_\lambda(V)])=\lim\limits_{r\to+\infty}{\rm FPdim}_{A_r}([\mathbb{S}_\lambda(V)])=\lim\limits_{n\to+\infty}\rho_{k, \lambda}(n)=\dim \mathbb{S}_\lambda(V).$$

    Notice that the linear operators $\{\widehat{[X_\lambda]}\}_\lambda$ on $QH^*(Gr(k, n))|_{q=1}$ are commutative. In fact, they can be simultaneously diagonalized  \cite[section 11]{Riet} (with respect to the common basis $\{\sigma_I\}$ therein).
   Hence, the following equalities of spectral radius $\rho_n(\cdot)$ for linear operators on $QH^*(Gr(k, n))|_{q=1}$ hold.
   $$\rho_n(a\widehat{[X_\lambda]}+b\widehat{[X_\mu]})=a\rho_n (\widehat{[X_\lambda]})+b\rho_n(\widehat{[X_\mu]});\quad
    \rho_n (\widehat{[X_\lambda]} \widehat{[X_\mu]})=\rho_n (\widehat{[X_\lambda]})\rho_n (\widehat{[X_\mu]}).$$
   Together with  Proposition \ref{Witteniso}, this shows  that
     ${\rm FPdim}_{A_r}: A_r\to \mathbb{C}$ is a ring homomorphism for any $r$.
Hence, ${\rm FPd}^\bullet: A\to \mathbb{C}$ is a $\mathbb{Z}_+^\bullet$-ring homomorphism. (Here we notice that
 the way of defining ${\rm FPd}^\bullet: A\to \mathbb{C}$ by the linear extension of the map ${\rm FPd}^\bullet:B\to \mathbb{C}$ is consistent with the way obtained by taking the limit   $\lim\limits_{n\to +\infty}\rho_n(\Phi^{-1}_r(\alpha))$ for any $\alpha\in A$.)
  \end{proof}

\section{Appendix: Proof of Theorem \ref{thmmain2}}

\begin{lemma}\label{lemTay}
   Let $a>b> 0$, and define     $f: \mathbb{R}^+\to\mathbb{R};   f(x):= \frac{\sin \frac{a}{x} }{\sin \frac{b}{x}}$.
   Then we have $f'(x)>0$ for any $x> {a\over \pi}$.
\end{lemma}
 \begin{proof}
By direct calculations, we have
$$f'(x)={g(x)\over x^2 \left(\sin {b\over x} \right)^2}, \,\,\mbox{ where }\,\, g(x)={-a \cos {a\over x}\sin {b\over x}+ b\sin {a\over x}  \cos {b\over x}}.$$
Therefore $\displaystyle g'(x)={\sin {a\over x}\sin {b\over x}\over x^2}\cdot (-a^2+b^2)<0$ for any $x>{a\over \pi}$.
 Notice $\lim\limits_{x\rightarrow +\infty} g(x)=0$. It follows that  $g(x)>0$ and hence  $f'(x)> 0$ for any  $x\in \left({a\over \pi}, +\infty \right)$.
\end{proof}

\begin{proof}[Proof of Theorem \ref{thmmain2}]
    Notice that $\rho_{k, \lambda}(x)=\prod_{r=1}^{|\lambda|} C_r(x)$  with   $C_r(x)={\sin {a_{r}\pi\over x}\over \sin {b_{r}\pi\over x}}$. Clearly  $a_r\geq b_r>0$ for any $r$, and $\max\{a_r\mid 1\leq r\leq |\lambda|\}=k+\lambda_1-1$.

    To prove statement (1),   following the proof of \cite[Lemma 5.1]{ESSSW},  we use the elementary inequalities $x-{x^3\over 6}\leq \sin x\leq x$ for $x\geq 0$. It follows that
$$ C_r(n)={\sin{a_r\pi\over n}\over  \sin{b_r\pi \over n}} \geq {{a_r\pi\over n} \left(1-{1\over 6}\cdot \left({a_r\pi \over n}\right)^2 \right)\over {b_r\pi\over n}}= \frac {a_r}{b_r} \left(1- \frac{\pi^2 a_r^2}{6n^2} \right),$$
where $a_r\leq k+\lambda_1-1\leq n-1$ so that $1- \frac{\pi^2 a_r^2}{6n^2}>0$. Thus we have
  $$\rho_{k, \lambda}(n)=\prod_rC_r(n)\geq \prod_r {a_r\over b_r}\prod_r \left(1- \frac{\pi^2 a_r^2}{6n^2} \right)=\dim \mathbb{S}_{\lambda}(V)\prod_{(i, j)\in \lambda} \left(1- \frac{\pi^2 (k-i+j)^2}{6n^2} \right).$$

  To prove statement (2), we notice that for any  $x>k+\lambda_1-1$,
      $C_r(x)>0$; moreover, we have  $C'_r(x)>0$ whenever $a_r>b_r$ by Lemma \ref{lemTay}.
  Since $\lambda_i\neq \lambda_j$ for some $i\neq j$, $a_r>b_r$ does hold for some $r$. It follows that
    $$\rho_{k, \lambda}'(x)=\rho_{k, \lambda}(x)\sum_{r=1}^{|\lambda|}{C_r'(x)\over C_r(x)}>0 \mbox{ for any } x>k+\lambda_1-1,$$   where   $C_s(x)=1$ and $C'_s(x)=0$ whenever $a_s=b_s$.  That is, the first half   holds.
     \begin{align*}
             \rho_{k, \lambda}''(x)&=\rho_{k,\lambda}(x)\sum_{i, j}{C_i'(x)C_j'(x)\over C_i(x)C_j(x)}+\rho_{k,n}(x)\sum_r{C_r''(x)\over C_r(x)}-\rho_{k, n}(x)\sum_r \left({C_r'(x)\over C_r(x)}\right)^2\\
             &=\rho_{k, n}(x)\left(\sum_{i\neq j}{C_i'(x)C_j'(x)\over C_i(x)C_j(x)}+ \sum_r{C_r''(x)\over C_r(x)}\right).\\
        \end{align*}
    By Taylor expansion around $x=+\infty$, we have
          $${C'_r(x)\over C_r (x)}={ (-a_r^2+b_r^2)\pi^2\over 3 x^3}+o \left({1\over x^4}\right),\quad {C''_r(x)\over C_r (x)}={(-a_r^2+b_r^2)\pi^2\over x^4}+o\left({1\over x^5} \right).$$
          Hence, $$x^4\cdot {\rho_{k, \lambda}''(x)\over \rho_{k, n}(x)}= \sum_{r}(-a_r^2+b_r^2)\pi^2+o\left({1\over x} \right),$$
   where the leading term is negative. Thus $\rho_{k, n}''(x)<0$ for sufficiently large $x$.
\end{proof}
   \section*{Acknowledgements}
The  authors would like to thank Hongjia Chen, Leonardo C. Mihalcea, Kyoji Saito, Peng Shan, Wei Yuan and Ming Zhang for helpful discussions. The authors would also like to thank an anonymous referee for bringing our attention to references on the literature on the theory of rigid $C^*$-tensor categories.
C. Li is  supported      by NSFC Grants 11822113, 11831017 and 11771455.

\end{document}